%% file: timeresponse.tex
\documentclass[a4paper,12pt]{article}

\usepackage{epsfig}
\usepackage{amsmath}
\usepackage{amssymb}
\usepackage{graphicx}
\usepackage{subfigure}
\usepackage{float}
\usepackage{cite}

\graphicspath{{}}

\newtheorem{algorithm1}{Algorithm}
\newtheorem{theorem}{Theorem}[section]
\newtheorem{definition}{Definition}[section]

\newtheorem{assumption}{Assumption}[section]

\newtheorem{remark}{Remark}[section]
\newtheorem{problem}{Problem}[section]

\newcommand{\ree}{\ensuremath{\mathbb{R}}}
\newcommand{\R}{\ensuremath{\mathbb{R}}}
\newcommand{\F}{\ensuremath{\mathcal{F}}}
\newcommand{\Fo}{\mathcal{F}_{\text{original}}}
\newcommand{\Fa}{{\mathcal F}_{\text{approx}}}


\title{\bf Linear Control of Time-Domain\\ Constrained Systems}

\begin{document}

\author{W.H.T.M.~Aangenent$^1$, W.P.M.H.~Heemels$^1$, M.J.G.~van~de~Molengraft$^1$,\\
D.~Henrion$^{2,3,4}$, M.~Steinbuch$^1$}

\footnotetext[1]{Department of Mechanical Engineering,
Eindhoven University of Technology,
P.O. Box 513, 5600 MB Eindhoven, The Netherlands}
\footnotetext[2]{CNRS; LAAS; 7 avenue du colonel Roche, F-31077 Toulouse; France}
\footnotetext[3]{Universit\'e de Toulouse; UPS, INSA, INP, ISAE; UT1, UTM, LAAS; F-31077 Toulouse; France}
\footnotetext[4]{Faculty of Electrical Engineering, Czech Technical University in Prague,
Technick\'a 2, CZ-16626 Prague, Czech Republic}

\maketitle

\begin{abstract}                          
This paper presents a general framework for the design of linear controllers for linear systems subject to time-domain constraints. The design framework exploits sums-of-squares techniques to incorporate the time-domain constraints on closed-loop signals and leads to conditions in terms of linear matrix inequalities (LMIs). This control design framework offers, in addition to constraint satisfaction, also the possibility of including an optimization objective that can be used to minimize steady state (tracking) errors, to decrease the settling time, to reduce overshoot and so on. The effectiveness of the framework is shown via a numerical example.
\end{abstract}

\begin{center}
\small
{\bf Keywords}: Linear Systems; Constrained systems; Polynomial methods;\\ LMIs; Motion systems.
\end{center}

\section{Introduction}\label{sec:Introduction}
The transient response to reference commands or disturbance inputs is an important performance qualifier in many control systems. Unfortunately, most control design strategies cannot cope directly with requirements on time-domain signals such as actuator amplitude
or rate limits, no output signal overshoot or undershoot, trajectory planning
constraints and so on. Especially in the continuous-time case, there are hardly any systematic controller design methods to enforce time-domain constraints on e.g. tracking errors and control inputs.

In the discrete-time case, model predictive control (MPC) (see e.g.~the surveys \cite{Mayne2000,Qin2003,Garcia1989}) is a widely used technique to cope with constraints on inputs and states. In MPC a control action is prescribed that is obtained by solving a finite or infinite horizon optimization problem that can incorporate input, state and output constraints in a direct manner. A drawback of predictive control concepts and online optimization-based methods in general is that they require a high computational effort with the consequence that they cannot be implemented on fast motion systems where high sampling rates are required, typically in the order of several kHz. Explicit MPC \cite{Bemporad2002,Tondel2003,Borrelli2003,BemporadHeemels2002} might offer an appealing solution as it precomputes a piecewise affine state feedback for discrete-time systems off-line. Still,  the explicit control law often leads to a complex description consisting of many affine feedbacks, which also cannot realize the high sampling rates typically needed for motion systems of considerable size,  although recent research is focussed on decreasing the implementation complexity of MPC, see for instance  \cite{Johansen2006,Grancharova2005,Grieder2005,Lazar2008,Kvasnica2008} and the references therein. An alternative approach with strong ties to MPC  is based on so-called reference governors, see e.g. \cite{Bemporad1997,GilKol_Aut_02} and the references therein. A reference governor is a nonlinear device that is added to a primal controller, which functions well in the absence of constraints. The reference governor modifies the reference signal supplied to the primal controller in order to enforce the input and state constraints. This approach suffers from the mentioned drawbacks in MPC to some extent  as well, but has the advantage that the  reference modifications are often needed at a lower sampling frequency than the updates of the primal control loop. A major difference with the method presented in this paper is that the overall control systems in case of reference governors become nonlinear devices modifying the supplied reference, while the method in this paper aims at designing {\em linear} controllers that satisfy the time-domain constraints without any modification of the references or disturbances.

Besides these predictive control methods, that are typically suited for a discrete-time context, there are only a few methods available in the literature that can directly synthesize controllers incorporating time-domain constraints in the continuous-time setting. For instance in the case of input constraints, \cite{Goebel2007,Heemels1998} consider the linear quadratic regulator problem with positivity constraints on the input, while various control problems with amplitude and rate constraints on the input signal are solved in \cite{Saberi2000}. The latter line of work has also been extended to stabilization and output regulation problems with amplitude and rate constraints on certain output variables, see, e.g., \cite{Saberi2002}. Other methods exist that actually allow the control output to saturate such as, for instance, the usage of anti-windup schemes \cite{Franklin2002,Tarbouriech2009,Tarbouriech2007} or {LQR/LQG} control methods \cite{Gokcek2001,ChiKab_TAC_10}. These methods, however, do not enforce constraint satisfaction but rather guarantee stability or recover performance despite the saturation nonlinearity in the loop. The above mentioned techniques cannot handle {\em time-varying} constraints, and,  except for \cite{Saberi2002}, state or output constraints are not considered either. In addition, all the above mentioned techniques result in general in {nonlinear} controllers.

As already briefly mentioned, in this paper the objective is to derive a design method for {\em linear} controllers that incorporate possibly {\em time-varying} time-domain constraints on all closed-loop signals (inputs, states and outputs). Within this context, a commonly used method to capture the essence of time-domain specifications is the reformulation into frequency domain requirements \cite{Doyle1992}. Unfortunately, such reformulations are in general either approximate, conservative or both.

A methodology to enforce time-domain constraints on the input and output of a continuous-time linear control system is presented recently in \cite{Henrion2005}, where linear matrix inequality (LMI) techniques are used to synthesize a fixed order { linear} controller that satisfies the constraints. This is done in a polynomial setting in the sense that a controller is designed according to the well-known pole placement method using the Diophantine equation. This method allows the design of a controller that results in a closed-loop transfer function with prescribed pole locations, either exact, or within an admissible region of the complex plane. In \cite{Henrion2005}, all controllers with the prescribed pole locations are characterized using the Youla-Ku\v{c}era parametrization \cite{Kucera1994}. Next the degrees of freedom of the Youla-Ku\v{c}era parametrization are used to enforce certain time-domain constraints, such as bounds on the input amplitude and output overshoot, exploiting  sums-of-squares techniques. 
Unfortunately, the approach in \cite{Henrion2005} is limited to the assignment of distinct strictly negative {\em real} closed-loop poles, which is a severe restriction in the case of many practical situations such as, for instance, lightly damped systems. As a consequence, there is a strong need for a general framework encompassing arbitrary closed-loop pole placement. The development of such a framework is the main purpose of this paper.

In particular, we propose an extension to the method in \cite{Henrion2005}, which leads to a general design framework based on sums-of-squares LMI techniques and we show indeed that the resulting linear controller satisfies the time-domain constraints on closed-loop signals, even when complex conjugate poles are assigned. This framework is based on two relaxations. One of these relaxations, of which a preliminary version was presented by the authors in \cite{Aangenent2009}, can solve the constrained control problem at hand with arbitrary accuracy and still lead to LMIs. In addition to constraint satisfaction, we will also include an objective function in the convex programming problem that can be used to minimize steady state (tracking) errors, to decrease the settling time, to reduce the overshoot and so on. As a consequence, the ideas presented in this paper will lead to a general design framework for optimized linear controllers with guarantees regarding constraint satisfaction.

The organization of the paper is as follows. The proposed methodology from \cite{Henrion2005} is briefly reviewed in Section~\ref{sec:henrionreview}. The extension to complex conjugate poles is treated in Section~\ref{sec:extcompoles}, which includes the main results. Section~\ref{sec:propconroldes} discusses the proposed control design method, and Section~\ref{sec:example} provides an illustrative example. Finally, the conclusions are stated in Section~\ref{sec:henryconcl}.\\

\section{Methodology involving real poles}\label{sec:henrionreview}
In \cite{Henrion2005} a method is presented to incorporate time-domain constraints on input and output signals of a linear system. It is shown that finding a controller of fixed order that satisfies these constraints boils down to solving a set of LMIs. In this section, we shortly review this procedure for completeness and self-containedness.
\subsection{Youla-Ku\v{c}era parametrization}\label{sec:youladio}
Consider the control system depicted in Fig.~\ref{fig:blockschemetd} with a linear single-input-single-output plant $P$ given by the strictly proper transfer function
\begin{equation}
P(s)=\frac{b(s)}{a(s)},
\label{eq:timdomcoplant}\end{equation}
where $a(s)$ and $b(s)$ are polynomials in the Laplace variable $s$.
\begin{figure}[thb]\begin{center}
\includegraphics[angle=0]{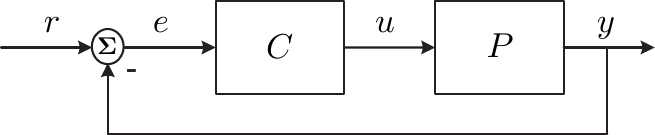} 
\caption{Block diagram of the closed-loop system with controller $C$, plant $P$, and reference signal $r$, control output signal $u$, and output signal $y$.}
\label{fig:blockschemetd}
\end{center}\end{figure}
The controller $C$, which is to be designed, is described accordingly by
\begin{equation}
C(s)=\frac{d(s)}{c(s)},
\label{eq:tdccondc}\end{equation}
resulting in the complementary sensitivity given by
\begin{equation}
T(s)=\frac{y(s)}{r(s)}=\frac{b(s)d(s)}{a(s)c(s)+b(s)d(s)}.
\label{eq:closedloop12}\end{equation}
If $a(s)$ and $b(s)$ are coprime (i.e., their greatest common divisor is 1), then arbitrary pole placement can be achieved by designing the corresponding controller polynomials. This is done by solving the polynomial Diophantine equation
\begin{equation}
a(s)c(s)+b(s)d(s)=z(s),
\label{eq:Diophantine}\end{equation}
where $z(s)\!=\!(s+p_1)(s+p_2)\ldots(s+p_n)$ is the polynomial with given roots $-p_1,\ldots, -p_n$, which are the desired poles of the closed-loop system. There are infinitely many solutions $(c(s),d(s))$ to \eqref{eq:Diophantine}, but there is a unique solution pair $\left(c_0(s), d_0(s)\right)$ such that $\deg d_0(s)\!<\!\deg a(s)$. In this case we have that $d_0(s)$ is of minimal degree and as such, $(c_0(s),d_0(s))$ is called the $d$-minimal solution pair. All possible solutions to the Diophantine equation can then be written as
\begin{equation}\begin{split}
c(s)&=c_0(s)+b(s)q(s),\\
d(s)&=d_0(s)-a(s)q(s),
\label{eq:allcontrollersparam}\end{split}\end{equation}
where $q(s)$ is an arbitrary polynomial such that $c_0(s)+b(s)q(s)$ is non-zero. This polynomial, called the Youla-Ku\v{c}era parameter \cite{Francis1987}, creates extra freedom in the design of the controller. While the closed-loop poles are invariant for any choice of the Youla-Ku\v{c}era parameter, the Youla-Ku\v{c}era parameter enables placement of closed-loop zeros to alter the response. Only {\em proper} controllers are considered and therefore there is a degree constraint on $q(s)$. Since the plant was assumed to be strictly proper, and under the additional assumption that $\deg z(s)\!\geq\! 2 \deg a(s)\!-\!1$ (to enable arbitrary pole placement with proper controllers), this constraint is given as in \cite{Kucera1999} by
\begin{equation}
\deg q(s)\leq \deg z(s)-2\deg a(s).
\label{eq:qdegreeconstaint}\end{equation}
The extra freedom in the control design parameterized by $q(s)$ satisfying \eqref{eq:qdegreeconstaint} can now be used to satisfy additional time-domain constraints as will be explained in the next section.

\subsection{A positive polynomial formulation of time-domain constraints}\label{sec:henlmipospol}
We will explain the procedure in \cite{Henrion2005} using the typical example of constraints on the step response. Hence, we consider the response $y$  to a step input ($r(s)\!=\!\frac{1}{s}$). The Laplace transform of the closed-loop system's output (assuming zero initial conditions) is then given by
\begin{equation}
y(s)=\frac{1}{s}\frac{b(s)d(s)}{z(s)}=\frac{1}{s}\frac{b(s)d_0(s)}{z(s)}-\frac{1}{s}\frac{a(s)b(s)}{z(s)}q(s).
\label{eq:laplaceoutput1}\end{equation}
At this point of the control design a \emph{restrictive assumption} was made \cite{Henrion2005}, namely
\begin{assumption}
All the assigned poles $-p_1,\ldots, -p_n$ are distinct strictly negative rational numbers.
\label{as:poleplace}\end{assumption}
Using this assumption and $z(s)\!=\!\prod_{i=1}^n(s\!+\!p_i)$ the partial fractional decomposition of \eqref{eq:laplaceoutput1} leads to
\begin{equation}
y(s,q)=\sum_{i=0}^n\frac{y_i(q)}{s+p_i},
\label{eq:partialfractionaldecom}\end{equation}
where $p_0\!=\!0$ and $y_i(q)$, $i\!=\!1,\ldots,n$ are appropriate coefficients following from the decomposition, which are influenced by the choice of the design parameter $q(s)\!=\!\sum_{i=0}^{d_q} q_is^i$. The coefficients $y_i(q)$ depend in an {\em affine} manner on the parameter $q = (q_0,q_1,\ldots,q_{d_q})$ in the sense that there exist matrices $A\in \ree^{(n+1)\times (n+1)}$, $B\in \ree^{(n+1)\times (q_d)}$ and a vector $b \in \ree^{n+1}$ such that
\begin{equation}
  A \begin{pmatrix}
    y_0(q) \\ y_1(q) \\ \vdots \\ y_{n+1}(q)
  \end{pmatrix} = B q^T + b.
\label{eq:affinedepyq}\end{equation}
This follows directly by comparing \eqref{eq:laplaceoutput1} and \eqref{eq:partialfractionaldecom}, and equating the coefficients of the powers of $s$ in the resulting numerator polynomials (see also \eqref{eq:linrelrealresex} below for an example). The corresponding time-domain signal is given by
\begin{equation}
y(t,q)=\sum_{i=0}^ny_i(q) e^{-p_it}.
\label{eq:polrepresentationtime}\end{equation}
Let $p_i\!=\!\frac{n_i}{d_i}$ be the ratios of the integers $n_i$ and $d_i$, and let $m$ denote the smallest positive number\footnote{In principle $m$ can be chosen to be any positive number that results in integer values of $\bar{p}_i$. However, it turns out that by choosing $m$ as the smallest possible positive number the order of the resulting polynomial optimization problem is the lowest.} of the denominators such that $p_i\!=\!\frac{\bar{p}_i}{m}$ for some positive integers $\bar{p}_i$, $i=0,1,\ldots,n$. This means that the time-domain output signal at time $t\!\in\!\R_+\!:=\![0,\infty)$ can now be expressed as the polynomial
\begin{equation}
y(\lambda,q)=\sum_{i=0}^ny_i(q) \lambda^{\bar{p}_i}
\label{eq:polrepresentationlamda}\end{equation}
in the indeterminate $\lambda\!=\!e^{-t/m}$. Obviously, $\lambda$ lies in the interval $[0,1]$ as $t\!\in\! {\mathbb R}_+$.
Suppose that the output $y(t,q)$ of the system needs to be bounded according to
\begin{equation}
y_{\text{min}}\leq y(t,q)\leq y_{\text{max}}\quad \forall~ t\in\R_+.
\label{eq:timedomaincs}\end{equation}
Formulation \eqref{eq:timedomaincs} is equivalent to enforcing the polynomial bound constraints
\begin{equation}
\left\{\begin{array}{l}P_1(q,\lambda):=y(\lambda,q)-y_{\text{min}}\geq0\\P_2(q,\lambda):=y_{\text{max}}-y(\lambda,q)\geq0\end{array}\right.\quad \forall~ \lambda\in[0,1],
\label{eq:positivepolycs}\end{equation}
where  $P_1$ and $P_2$ are polynomials in both $\lambda$ and $q$. This problem is a special case of the following  more general  problem  of minimizing a polynomial with polynomial constraints over a basic semialgebraic set.
\begin{definition}
A set $\mathcal D$ is called a basic semialgebraic set if it can be described as
\begin{equation}\begin{split}
{\mathcal D}\!&=\! \{ x\!\in\! \R^{n} \mid e_i(x)\! \geq\! 0, \ i\!=\!1,\ldots, M_e \text{ and } \\
&\quad\quad\quad f_j(x)\! =\! 0,\ j\!=\!1,\ldots,M_f \}
\label{eq:seiset}\end{split}\end{equation}
for certain polynomials $e_i\!:\! \R^{n_x}\! \rightarrow\! \R$, $i\!=\!1,\ldots, M_e$  and $f_j \!:\! \R^{n_x}\! \rightarrow \!\R$, $j\!=\!1,\ldots,M_f$.
\label{def:semialset}\end{definition}

\begin{problem}[Polynomial optimization problem] \label{pr:opt} Consider two variables $z\in \R^{n_z}$ and $x \in \R^{n_x}$ and let polynomials $g_i: \R^{n_z} \times \R^{n_x} \rightarrow \R$, $i=1,\ldots, M_g$, and $p:\R^{n_z} \rightarrow \R$ be given. Moreover, let a collection of  basic semialgebraic sets ${\mathcal D_l}\subseteq \R^{n_x},~l=0,\ldots,N$ be given. A (robust) polynomial optimization problem according to this data is given by
\begin{equation}
\begin{array}{ll}\underset{z}{\text{min}}&p(z)\\\text{s.t.}&\begin{array}{l}g_i(z,x)\!\geq\! 0,\quad ~i\!=\!1,\ldots,M_g\end{array}\quad\! \forall~ x\!\in\!\overset{N}{\underset{l=0}{\bigcup}}\mathcal{D}_l.
\end{array}\label{eq:generalpolprob}\end{equation}
\end{problem}
Indeed, \eqref{eq:positivepolycs} can now be written in the form of Problem~\ref{pr:opt} by taking
$z\!=\!q$, $x\!=\!\lambda$, $M_g \!=\!2,~N\!=\!0$, $p(z)\!=\!0$, $g_1(z,x)\!=\!P_1(q,\lambda)$, $g_2(z,x)\!=\!P_2(q,\lambda)$, and $\mathcal{D}_0=\{\lambda\!\in\!\R\mid0\!\leq\!\lambda\!\leq\!1\}$.
Although the bounds $y_{\text{min}}$ and $y_{\text{max}}$ in \eqref{eq:positivepolycs} are chosen to be constants for illustrating purposes, they can also be selected as polynomials in $\lambda$, i.e.,~in the form $y_{\text{min}}(\lambda)$ and $y_{\text{max}}(\lambda)$ without any complications.
In this case the bounds in \eqref{eq:timedomaincs} become {\em time-varying}. Univariate positive polynomial constraints (meaning polynomials in only one variable), such as \eqref{eq:positivepolycs} with $\lambda\in{\mathcal D}_0=[0,1] \subseteq \R$, can be transformed into LMI conditions, see \cite{Henrion2005} for the details.
Once we transformed the design problem into a polynomial optimization problem as formulated in Problem~\ref{pr:opt}, there are appropriate tools available for solving the problem. Therefore, we restrict ourselves to the transformation of the constrained control problems at hand into manifestations of Problem~\ref{pr:opt}.

The approach discussed in this section is not restricted to bounding only the output of a system. Indeed, by using the appropriate transfer functions, any signal in the loop can be constrained. The control output $u$, for example, can be bounded using
\begin{equation}
u(s)=r(s)\frac{a(s)d_0(s)}{z(s)}-r(s)\frac{a^2(s)}{z(s)}q(s)
\end{equation}
in addition to, or instead of \eqref{eq:laplaceoutput1}. Also, other Laplace transformable inputs or disturbances  can be used as long as the poles of the Laplace transform of the corresponding signal are distinct strictly negative rational numbers and differ from the closed-loop poles $p_i$. In case the disturbance signal is a (filtered) random process, the method cannot be applied as is. However, a possible extension can be to bound the infinity- or 2-norm of the suitably weighted process sensitivity (or other relevant transfer functions) via e.g.
\begin{equation}
||V(s)\frac{b(s)c(s)}{a(s)c(s)+b(s)d(s)}W(S)||_{2/\infty}\leq1
\end{equation}
This way the knowledge of stochastic disturbances can be used to shape the relevant sensitivity functions to achieve desired disturbance reduction.\\
A combination of requirements on different reference signals can easily be handled at the cost of increasing the size of the set of LMIs. As for LMIs there are efficient solvers available, e.g. \cite{Sturm1999}, transforming the problem at hand into Problem~\ref{pr:opt} provides an effective solution. The problem derived in this section  was only a feasibility problem (the cost criterion $p(z)$ was chosen to be $0$). In Section~\ref{sec:propconroldes} below, we will also provide relevant choices for the cost criterion that next to  satisfaction of  the time-domain constraints also provides additional desirable properties of the constructed controller.

\section{Problem formulation: the complex poles case}\label{sec:extcompoles}
The polynomial representation \eqref{eq:polrepresentationtime}, as derived in \cite{Henrion2005}, of the time response of a linear system to a Laplace transformable input is unfortunately only possible when strictly negative rational closed-loop poles are assigned (see Assumption~\ref{as:poleplace}). However, in many cases the assignment of purely real poles can be undesirable, especially in lightly damped systems such as most motion systems.  Furthermore, many reference signals have Laplace transforms with complex poles. If, for instance, a sinusoid is used as the reference signal instead of a step, the Laplace transform is given by $r(s)=\frac{\omega}{s^2+\omega^2}$ resulting in complex poles in the system's response. Therefore, such reference signals cannot be handled by the approach from \cite{Henrion2005}.  The main objective of this paper is to present a solution to the linear control design problem for time-domain constrained systems of which the Laplace transforms of the closed-loop responses may contain complex roots.\\
When we allow both distinct real and complex poles to be present in the closed-loop transfer function $T(s)$ and/or the Laplace transform of the reference signal $r(s)$, the Laplace transform of the system's output can be decomposed as the partial fractional decomposition
\begin{equation}\begin{split}
y(s)=&\sum_{i=0}^{n_r}\frac{y_i}{s+p_i}\\
&+\sum_{i=n_r+1}^{n_r+n_c/2+1}\frac{y_i}{s+\alpha_i+j\beta_i}+\frac{y_i^*}{s+\alpha_i-j\beta_i},
\label{eq:partialfractionaldecomrealcomp}\end{split}\end{equation}
where $n_r$ and $n_c$ denote the number of real and complex poles, respectively, $-p_i$, $i=1,\ldots,n_r$ are the locations of the real poles, $-\alpha_i\pm j\beta_i$, $i=n_r+n_c/2+1$ are the locations of complex conjugate pairs of poles, and $y_i$ are the possibly complex coefficients (with complex conjugate $y_i^*$) that affinely depend on the design parameter $q$, see \eqref{eq:affinedepyq} (we omitted this dependence on $q$ for ease of exposition). To enforce stability, we again assume that the assigned closed-loop poles have strictly negative real part. The corresponding time-domain signal is then described by
\begin{equation}
y(t)=\sum_{i=0}^{n_r}y_ie^{-p_it}\quad+\sum_{i=n_r+1}^{n_r+n_c/2+1}(y_ie^{-j\beta_it} +y_i^*e^{j\beta_it})e^{-\alpha_i t}.
\label{eq:exacttimecompsub1}
\end{equation}
As before, we use the following assumption
\begin{assumption}
$p_i$, $\alpha_i$, and $\beta_i$ are rational numbers.
\label{ass:ratnum}\end{assumption}
We denote $p_i\!=\!\frac{\bar{p}_i}{m}$, $\alpha_i\!=\!\frac{\bar{\alpha}_i}{m} $, $\tau\!=\!\frac{t}{m}$ where $m$ is the smallest positive number (not necessarily an integer) of $p_i$ and $\alpha_i$ such that $\bar{p_i}$ and $\bar{\alpha}_i$ can be taken as integers. We denote $\beta_i\!=\!\frac{\theta\bar{\beta}_i}{m}$ for a number $\theta$ (not necessarily an integer) such that $\bar{\beta}_i$ can be taken as integer as well. For guidelines how to choose $\theta$, see Remark~\ref{remark:ding} below.  Furthermore, let
\begin{equation}
\lambda=e^{-\tau}.
\label{eq:lambda3pol}\end{equation}
Using Euler's formula $e^{j\phi}=\cos(\phi)+j\sin(\phi)$ and decomposing the complex coefficients as $y_i\!=\!a_i\!+\!jb_i$, $y_i^*\!=\!a_i\!-\!jb_i$, yields
\begin{equation}\begin{split}
y(t)=&\sum_{i=0}^{n_r}y_i\lambda^{\overline{p}_i}\\
&+\sum_{i=n_r+1}^{n_r+n_c/2+1}\left(a_i2\cos(\overline{\beta}_i \theta\tau)+b_i2\sin(\overline{\beta}_i \theta\tau)\right)\lambda^{\overline{\alpha}_i}.
\label{eq:exacttimecomp}
\end{split}\end{equation}
Obviously, the terms involving the complex poles are non-polynomial in the indeterminate $\lambda$ because of the presence of   $\cos(\overline{\beta}_i \theta\tau)$ and $\sin(\overline{\beta}_i \theta\tau)$, which make it impossible to directly use the positive polynomial approach in Section~\ref{sec:henrionreview} to bound the output as in \eqref{eq:timedomaincs}. Although the parameters $\overline{\alpha}_i$, $\overline{\beta}_i$ and $\overline{p}_i$ are fixed as a result of the pole placement (and the choice of $m$ and $\theta$), there still is freedom in the choice for the coefficients $a_i, b_i$, which depend on the coefficients $q=(q_0,\ldots,q_{d_q})$ in the Youla-Ku\v{c}era parameter $q(s)$.  We propose two relaxations to determine the values $y_i,a_i, b_i$ via polynomial optimization problems to shape the time response $y(t)$, thereby overcoming the limitations in \cite{Henrion2005}. The first approach is based on an exponential bound relaxation that results in univariate polynomials. This method has the advantage that it results in a simple polynomial optimization problem, but introduces some conservatism. The second method proposes a multivariate polynomial relaxation that leads to polynomial problems as in Problem~\ref{pr:opt}, while the conservatism can be made arbitrarily small.  Both methods result in polynomial optimization problems of the type as in Problem \ref{pr:opt} that can be solved using LMIs.

\subsection{Exponential bounds relaxation}\label{sec:expbrelaxth}
To resolve the problem  induced by the presence of the products $\cos(\overline{\beta}_i \theta\tau)\lambda^{\overline{\alpha}_i}$ and $\sin(\overline{\beta}_i \theta\tau)\lambda^{\overline{\alpha}_i}$ in \eqref{eq:exacttimecomp} we relax the problem by using the fact that
\begin{equation}
\cos(\overline{\beta}_i\theta\tau),~\sin(\overline{\beta}_i \theta\tau)\in[-1,1]\quad\forall \tau\in\R,\\
\end{equation}
and instead of the exact time-response \eqref{eq:exacttimecomp}, we consider
\begin{equation}\begin{split}
y_{\text{upper}}(\lambda)&=\sum_{i=0}^{n_r}y_i\lambda^{\overline{p}_i}+\sum_{i=n_r+1}^{n_r+n_c/2+1} \left(2|a_i|+2|b_i|\right)\lambda^{\overline{\alpha}_i},\\
y_{\text{lower}}(\lambda)&=\sum_{i=0}^{n_r}y_i\lambda^{\overline{p}_i}-\sum_{i=n_r+1}^{n_r+n_c/2+1} \left(2|a_i|+2|b_i|\right)\lambda^{\overline{\alpha}_i}.
\end{split}\label{eq:boundstimecomp}\end{equation}
In contrast to \eqref{eq:exacttimecomp}, these exponential bounds on the closed-loop time response are univariate polynomials in the indeterminate $\lambda\!=\!e^{-\tau}$ (if $q$ is fixed) and can be bounded by specified polynomials $g_u(\lambda)$ and $g_l(\lambda)$ via the polynomial non-negativity constraints
\begin{equation}
\begin{array}{ll}P_3(q,\lambda):=g_u(\lambda)-y_{\text{upper}} (\lambda)&\geq0\\P_4(q,\lambda):=y_{\text{lower}}(\lambda)-g_l(\lambda)&\geq0\end{array}\quad\forall\lambda\in[0,1], \label{eq:positivepolyexpboundsa}\end{equation}
where we included the explicit dependence of $y_i,a_i$ and $b_i$ on $q$ again. The constraints \eqref{eq:positivepolyexpboundsa} cannot straightforwardly be cast in the form Problem~\ref{pr:opt} because of the nonlinear operator $|\cdot|$, which is present in these equations. However, each of the two nonlinear inequality constraints $P_3$ and $P_4$ can be expressed as $2^{n_c+1}$ equivalent polynomial inequality constraints $\tilde{P}_3$ and $\tilde{P}_4$ (2 inequalities for each absolute value expression). Enforcing non-negativity of \eqref{eq:positivepolyexpboundsa} on the interval $\lambda\in[0,1]$ 
is then again a special case of Problem \ref{pr:opt} with $z\!=\!q$, $x\!=\!\lambda$, $M_g\!=\!2,~M_h\!=\!0$, $N\!=\!0$, $p(z,x)\!=\!0$, $g_1(z,x)\!=\!\tilde{P}_3(q,\lambda)$, $g_2(z,x)\!=\!\tilde{P}_4(q,\lambda)$, and $\mathcal{D}_0=\{\lambda\!\in\!\R\mid0\!\leq\!\lambda\!\leq\!1\}$. Therefore, it is possible to determine the values $q=(q_0,\ldots,q_{d_q})$ such that the upper and lower bounds \eqref{eq:boundstimecomp} of the closed-loop time response are bounded by $g_u(\lambda)$ (e.g. $g_u(\lambda)= y_{max}$) and $g_l(\lambda)$ (e.g. $g_u(\lambda)= y_{min}$) via a polynomial optimization problem. 
The exponential bounds relaxation does 
introduce some conservatism by using relaxation \eqref{eq:boundstimecomp} instead of the exact time-response \eqref{eq:exacttimecomp}. The second method presented next offers the possibility to render this conservatism arbitrary small. In other words, the second method can approximate the original time-domain constraints with arbitrary accuracy and still lead to polynomial optimization problems.

\subsection{Multivariate polynomial relaxation}\label{sec:mulpolrelax}
The time response \eqref{eq:exacttimecomp} is equivalent to
\begin{equation}\begin{split}
&y(t)=\sum_{i=0}^{n_r}y_i\lambda^{\overline{p}_i}
+\\
&\sum_{i=n_r+1}^{n_r+n_c/2+1}\left[ (a_i\!+\!jb_i)\left(\cos(\overline{\beta}_i\theta\tau)\!-\!j\sin(\overline{\beta}_i\theta\tau)\right)\right.\\ &\quad\quad\left.\!+\!(a_i\!-\!jb_i)\left(\cos(\overline{\beta}_i\theta\tau)\!+\!j\sin(\overline{\beta}_i\theta\tau)\right)\right]\lambda^{\overline{\alpha}_i}.
\label{eq:exacttimecomp2}
\end{split}\end{equation}
De Moivre's formula, which is closely related to Euler's formula and $(e^{j\phi})^n = e^{jn\phi}$,  states that for any $\phi\in\R$ and any integer $n\!\in\!\mathbb{Z}$
\begin{equation}
\left(\cos(\phi)+j\sin(\phi)\right)^n=\cos(n\phi)+j\sin(n\phi),
\label{eq:demoivre}
\end{equation}
and hence \eqref{eq:exacttimecomp2} is equal to
\begin{equation}\begin{split}
&y(t)=\sum_{i=0}^{n_r}y_i\lambda^{\overline{p}_i}
+\\
&\sum_{i=n_r+1}^{n_r+n_c/2+1}\left((a_i+jb_i)\left[\cos(\theta\tau)-j\sin(\theta\tau)\right]^{\overline{\beta}_i}+ \right.\\
&\quad\left.(a_i-jb_i)\left[\cos(\theta\tau) +j\sin(\theta\tau)\right]^{\overline{\beta}_i}\right)\lambda^{\overline{\alpha}_i}.
\label{eq:exacttimecomp3}
\end{split}\end{equation}
Appropriate polynomial functions $w_i:\R^2\rightarrow \R$ and $r_i:\R^2\rightarrow \R$, $i=n_r+1,\ldots,n_r+n_c/2+1$ in two variables can now be defined such that \eqref{eq:exacttimecomp3}, and thus the time response \eqref{eq:exacttimecomp}, can be written as
\begin{equation}\begin{split}
y(t)=&\sum_{i=0}^{n_r}y_i\lambda^{\overline{p}_i}+\\
&\sum_{i=n_r+1}^{n_r+n_c/2+1}\hspace*{-0.5cm}\left(a_i2w_i(\cos(\theta\tau),\sin(\theta\tau))\right.\\
&\left.+b_i2r_i(\cos(\theta\tau),\sin(\theta\tau))\right)\lambda^{\overline{\alpha}_i}.
\end{split}\label{eq:exacttimecompmoivre}\end{equation}
This proves the following theorem.
\begin{theorem}
Consider the closed-loop system \eqref{eq:closedloop12} and let $y$ be the response to a reference input $r$ and assume that the Laplace transform $y(s)$ of $y$ has only distinct poles such that \eqref{eq:partialfractionaldecomrealcomp} and Assumption \ref{ass:ratnum} hold. Then we have that
\begin{equation}
  \{ y(t) \mid t\in \R^+\} = \{y(u,v,\lambda) \mid (u,v,\lambda) \in {\mathcal F}_{\text{original}} \},
\end{equation}
where $y(u,v,\lambda)$ is given by the multivariate polynomial
\begin{equation}\begin{split}
y(u,v,\lambda)=&\sum_{i=0}^{n_r}y_i\lambda^{\bar{p}_i}\\
&+\sum_{i=n_r+1}^{n_r+n_c/2+1}\left(a_i2w_i(u,v)+b_i2r_i(u,v)\right)\lambda^{\bar{\alpha}_i}
\label{eq:exacttimecomp3pol}\end{split}\end{equation}
with $w_i:\R^2\rightarrow \R$ and $r_i:\R^2\rightarrow \R$, $i=n_r+1,\ldots,n_r+n_c/2+1$ polynomials as in \eqref{eq:exacttimecompmoivre}
and
%
\begin{equation}\begin{split}
\mathcal{F}_{\text{original}}:=&\left\{(u,v,\lambda)\in \R^3 \mid u=\cos(\theta\tau),~v=\sin(\theta\tau),\right.\\
&\quad\left.\lambda=e^{-\tau} \text{ for some } \tau\in\R^+\right\}.
\end{split}\end{equation}
\label{thm6}
\end{theorem}
\begin{proof}
  The reasoning before the formulation of the theorem revealed that $y(t)$ under the given assumptions is equal to \eqref{eq:exacttimecompsub1}, which can equivalently  be written as \eqref{eq:exacttimecompmoivre}, where $\lambda= e^{-\tau}$ and $\tau=\frac{t}{m}$. 
\end{proof}
Due to Theorem~\ref{thm6}, bounding the output as in \eqref{eq:timedomaincs} to the interval $[y_{\text{min}},y_{\text{max}}]$ is equivalent to enforcing the polynomial non-negativity constraints
\begin{equation}
\begin{array}{l}P_5(q,u,v,\lambda):=y(u,v,\lambda)-y_{\text{min}}\geq0,\\P_6(q,u,v,\lambda):=y_{\text{max}}-y(u,v,\lambda)\geq0\end{array} \label{eq:positivepolymultpolvar}\end{equation}
for all $(u,v,\lambda)\in\mathcal{F}_{\text{original}}$. Recall that $y(u,v,\lambda)$ depends on $q$ via $y_i,a_i$ and $b_i$. As we mentioned before, it is of interest to transform the linear constrained control problem into Problem~\ref{pr:opt}. The conditions \eqref{eq:positivepolymultpolvar} are not in this form due to the fact that $\mathcal{F}_{\text{original}}$ is \emph{not} a (finite union of) basic semialgebraic set(s) as in Definition \ref{def:semialset}. However, this set can be overapproximated by a finite union of basic semialgebraic sets in an arbitrarily close manner.
\begin{definition} \label{def:overapproximation}
 We call a set ${\mathcal F}_{\text{approx}}$ an {\em $\varepsilon$-close overapproximation} of $\mathcal{F}_{\text{original}}$ for some $\varepsilon>0$, if it satisfies the following three properties:
\begin{enumerate}
 \item ${\mathcal F}_{\text{approx}} = \bigcup_{l=0}^N {\mathcal F}_l$ for a finite collection of basic semialgebraic sets ${\mathcal F}_0,\ldots,{\mathcal F}_N$;
\item ${\mathcal F}_{\text{original}}\subseteq {\mathcal F}_{\text{approx}}$;
\item ${\mathcal F}_{\text{approx}} \subseteq {\mathcal F}_{\text{original}} +  \mathbb{B}_\varepsilon$, where $\mathbb{B}_\varepsilon:=\{(0,0,z) \mid -\varepsilon \leq z \leq \varepsilon $.
\end{enumerate}
\end{definition}

Hence, an{~$\varepsilon$-close overapproximation} of $\mathcal{F}_{\text{original}}$ contains  the set $\mathcal{F}_{\text{original}}$ as drawn by the white line in Figure~\ref{fig:cilinder} (for $\theta\!=\!1$), but it is  $\varepsilon$-close in the sense of property 3. Hence, for small $\varepsilon>0$, replacing $\Fo$ by $\Fa$ only results in small errors and all guarantees on $\Fa$ also apply to $\Fo$ due to property 2. Moreover, due to property 1 an $\varepsilon$-close overapproximation ${\mathcal F}_{\text{approx}}$ of $\mathcal{F}_{\text{original}}$ can be used to embed the polynomial constraints in \eqref{eq:positivepolymultpolvar} for all $(u,v,\lambda)\in\mathcal{F}_{\text{original}}$ into a version of the constraints in Problem~\ref{pr:opt}, where ${\mathcal D}_l={\mathcal F}_l,~l=0,\dots,N$.
\begin{figure}[thb]\begin{center}
\includegraphics[angle=0,width=8cm]{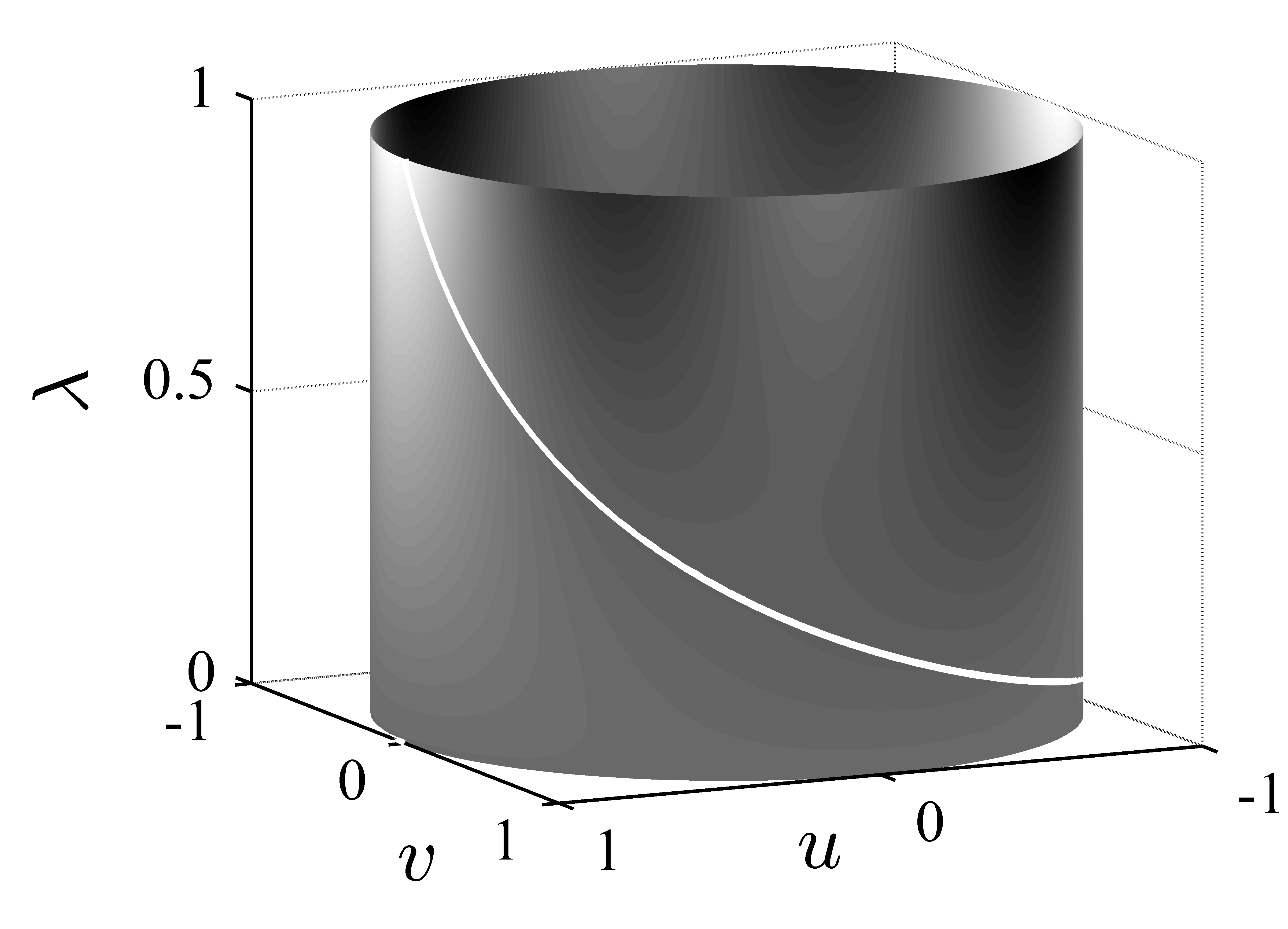} 
\caption{$\mathcal{F}_{\text{original}}$ (white line) drawn inside the cylinder given by $u^2+v^2=1$ and $0\leq \lambda \leq 1$.}
\label{fig:cilinder}
\end{center}\end{figure}

The following algorithm provides an algorithm that constructs for each desirable level of approximation $\varepsilon$ an $\varepsilon$-close overapproximation of $\Fo$. The basic idea of the algorithm is to overapproximate the $\Fo$-set by the union of basic semialgebraic sets ${\mathcal F}_l$, which are obtained by splitting the set $\Fo$ in the $\tau$-direction by considering intervals $I_l:=[\tau_{l},\tau_{l+1})$, $l=0,\ldots,N$, where  $0 =\tau_0<\tau_1<\ldots< \tau_{N+1}= \infty $. On each of these subintervals $I_l$ we approximate $e^{-\tau}$ by $\psi_l(\cos(\theta \tau), \sin(\theta \tau))$ using Fourier series, where $\psi_l:\R^2\rightarrow\R$ is a polynomial such that $|e^{-\tau} - \psi_l(\cos(\theta \tau), \sin(\theta \tau)) | \leq \varepsilon$ for all $\tau\in I_l$. Next to $\varepsilon$, the algorithm uses another parameter $0\!<\!T\!<\!\frac{2 \pi}{\theta}$, which indicates the desired length of the intervals $I_l$, $l=0,\ldots,N-1$ (although it will be modified such that all intervals have the same length).
%

\begin{algorithm1} \label{alg}
Let $0\!<\!  \varepsilon\!<\!1$ and $0\!<\!T\!<\!\frac{2 \pi}{\theta}$ be given.
\begin{description}
  \item[Step 1:] Define $N\!:=\! \lceil \frac{- \ln \varepsilon}{T} \rceil$ and $\tau_{N}\!:=\!-\ln \varepsilon$ and $\tau_{N+1}\!:=\! \infty$.
      \begin{equation}
\hspace*{-1cm}{\mathcal F}_N:=\{(u,v,\lambda) \in \R^3 \mid u^2+v^2=1 \text{ and } 0 \leq \lambda \leq \varepsilon \}.
\label{eq:FltotalN}\end{equation}
 \item[Step 2:] Divide the remaining interval $[0,\tau_{N})$ in $N$ subintervals of length $\bar{T}\!:=\! \frac{\tau_N}{N}\leq T <\frac{2 \pi}{\theta}$. $I_l\!:=\![\tau_l,\tau_{l+1})$ with $\tau_l \!=\! l\bar{T}$, $l\!=\!0,\ldots,N\!-\!1$.
 \item[Step 3:] For each $l\!=\!0,\ldots,N\!-\!1$ define a function $\phi_l:\R \rightarrow \R$ that satisfies:
\begin{itemize}
  \item $\phi_l$ is at least continuously differentiable, but preferably $m$ times continuously differentiable ($C^m$) for $m\in {\mathbb N}$ large;
\item $\phi_l$ is periodic with period $\frac{2 \pi}{\theta}$;
\item $\phi_l(\tau)=e^{-\tau}$ for all $\tau \in I_l$.
\end{itemize}
 \item[Step 4:] For each $l\!=\!0,\ldots,N\!-\!1$ compute the Fourier series approximation of $\phi_l$ of sufficiently high degree $K_l$ such that
\begin{equation} \label{eq:a}
|\phi_l(\tau) - \sum_{k=0}^{K_l} [a_k \cos(k\theta\tau) + b_k \sin (k\theta\tau) ] | \leq \varepsilon \text{ for all } \tau \in I_l,
\end{equation}
where $a_k,b_k$, $k\!=\!0,\ldots,K_l$ are the Fourier coefficients of $\phi_l$.
 \item[Step 5:] For each $l\!=\!0,\ldots,N\!-\!1$ use De Moivre's formula to rewrite $\sum_{k=0}^{K_l} [a_k \cos(k\theta\tau) + b_k \sin (k\theta\tau) ]$ obtained in the previous step as
\[\sum_{k=0}^{K_l}\sum_{i=0}^k c_{ki} (\cos(\theta\tau))^k(\sin (\theta\tau))^l =: \psi_l(\cos(\theta\tau),\sin(\theta\tau)),\]
where $\psi_l:\R^2\rightarrow \R$ is a polynomial of  degree equal to the degree of the Fourier series.
 \item[Step 6:]
 For each $l\!=\!0,\ldots,N\!-\!1$, define
\begin{multline}
  {\mathcal F}_l := \{(u,v,\lambda) \in \R^3 \mid -\varepsilon \leq \lambda-\psi_l(u,v) \leq \varepsilon \wedge \\
\hspace{1cm} u^2\!+\!v^2\!=\!1 \wedge (S_l\!-\!S_{l+1})u+(C_{l+1}\!-\!C_l)v+\\
S_{l+1}C_{l}\!-\!C_{l+1}S_l \leq 0\},
\label{eq:Fltotal}\end{multline}
where $C_l:=\cos(\theta \tau_{l})$, $S_l:=\sin(\theta \tau_l)$.
 \item[Step 7:] Take $\Fa = \bigcup_{l=0}^N {\mathcal F}_l.$
\end{description}
\end{algorithm1}

\begin{theorem}
For each $0< \varepsilon<1$ and $0<T<\frac{2 \pi}{\theta}$ Algorithm~\ref{alg} produces an $\varepsilon$-close overapproximation $\Fa$ of $\Fo$ in the sense of Definition~\ref{def:overapproximation}.
\end{theorem}
\begin{proof}
First of all, we write ${\mathcal F}_{\text{original}}$ as $\bigcup_{l=0}^N {\mathcal F}_{\text{original}, l}$ with  ${\mathcal F}_{\text{original}, l}  :=\{(\cos \theta\tau, \sin \theta\tau, e^{-\tau}) \mid \tau \in I_l \}$ for $l=0,1,\ldots,N$ as $\bigcup_{l=0}^N I_l = [0,\infty)$.
Step 1 considers the interval $I_N:= [\tau_{N},\tau_{N+1}) = [-\ln \varepsilon, \infty)$ for which it holds that $0\leq  e^{-\tau}  \leq \varepsilon$.  Hence, clearly ${\mathcal F}_{\text{original}, N}\subseteq {\mathcal F}_N$ and ${\mathcal F}_N\subseteq   {\mathcal F}_{\text{original}, N} + \mathbb{B}_\varepsilon$.
The construction of functions $\phi_l$ in step 3 is possible as $e^{-\tau}$ is continuous and the fact that $0<\bar{T}\leq T<\frac{2 \pi}{\theta}$. Hence, step 3 can always be taken, while the function $\phi_l$ can still be made $\frac{2 \pi}{\theta}$-periodic and continuously differentiable.
Step 4 can be realized, because the Fourier series converges uniformly to a continuously differentiable periodic function, see, e.g., \cite{Powers2006,Asmar2005}. Therefore, uniform convergence proves the existence of a finite $K_l$ such that \eqref{eq:a} holds.
Note now that for $(u,v,\lambda)\in\mathcal{F}_{\text{original},l}$ it holds that $u^2\!+\!v^2\!=\!1$. Obviously, $(\cos(\theta\tau),\sin(\theta\tau))$ for $\tau\!\in\! I_l\!=\![\tau_l,\tau_{l\!+\!1})$ lies in one of the half spaces generated by  
the straight line in $\R^2$ through the points  $(\cos(\theta\tau_l),\sin(\theta\tau_l))$ and  $(\cos(\theta\tau_{l\!+\!1}),\sin(\theta\tau_{l\!+\!1})$ given by
\begin{multline}
[\sin(\theta\tau_l)-\sin(\theta\tau_{l\!+\!1})]\underbrace{\cos(\theta\tau)}_u+[\cos(\theta\tau_{l\!+\!1})-\cos(\theta\tau_l)]\underbrace{\sin(\theta\tau)}_v\\+\sin(\theta\tau_{l\!+\!1})\cos(\theta\tau_l)-\cos(\theta\tau_{l\!+\!1})\sin(\theta\tau_l)=0,
\label{eq:taudivide}\end{multline}
see Figure \ref{fig:tau_divide}. In particular, for all $(u,v,\lambda)\in\mathcal{F}_{\text{original},l}$ it holds that
\begin{equation}
(S_l-S_{l+1})u+(C_{l+1}-C_l)v+S_{l+1}C_{l}-C_{l+1}S_l \leq 0
\label{eq:taudivedcylinder}\end{equation}
where $C_l\!:=\!\cos(\theta\tau_l)$, $S_l\!:=\!\sin(\theta\tau_l)$. Due to \eqref{eq:a}, step 5, and using the above observations, it holds that $\mathcal{F}_{\text{original},l}\!\subseteq\!\mathcal{F}_{l}$. Moreover, similar reasoning using \eqref{eq:a} shows that $\mathcal{F}_{l}\!\subseteq\!\mathcal{F}_{\text{original},l}\!+\!\mathbb{B}_\varepsilon$.
Hence, by taking $\Fa$ as the union of the resulting sets, i.e. $\Fa\!=\!\bigcup_{l=0}^N {\mathcal F}_l$, an $\varepsilon$-close overapproximation $\Fa$ of $\Fo$ is obtained. This  completes the proof.
\end{proof}
\begin{figure}[thb]\begin{center}
\includegraphics[angle=0]{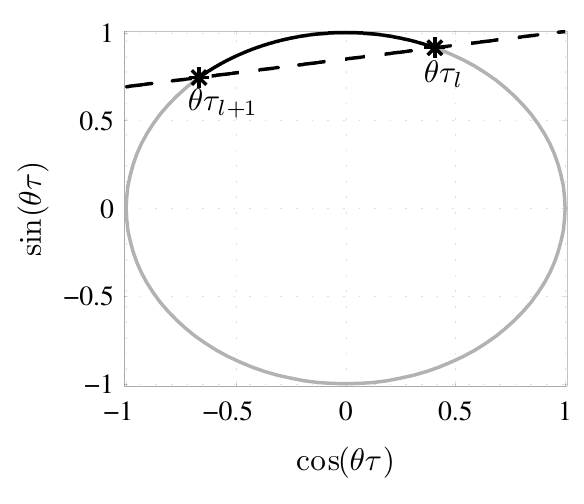}
\caption{Top-view of the cylinder $\{(u,v,\lambda) \in \R^3 \mid u^2\!+\!v^2\!=\!1, 0\leq\lambda\leq1\}$ (solid grey), the hyperplane $\{(u,v,\lambda) \in \R^3 \mid \eqref{eq:taudivide}\}$  (dashed), and the resulting part of the cylinder given by $\{(u,v,\lambda) \in \R^3 \mid u^2\!+\!v^2\!=\!1,\eqref{eq:taudivedcylinder}, 0\leq\lambda\leq1\}$ (solid black).}
\label{fig:tau_divide}
\end{center}\end{figure}

For instance, for $\theta\!=\!1, \varepsilon\!=\!e^{-1.5\pi}\!\approx\!0.009$ and $T=\bar{T}\!=\!0.75\pi$ the $\varepsilon$-close overapproximation $\Fa$ of $\Fo$  can be generated with $\tau_0=0$, $\tau_1=0.75\pi$, $\tau_2 = 1.5\pi$ and $\tau_3=\infty$ and polynomials
\begin{equation}\label{precomp}\begin{split}
\psi_0(u,v)\!&=\!0.398u\!-\!0.971v\!+\!0.616u^2\!-\!0.192uv\!+\!1.179v^2\\
&-\!0.015u^3\!+\!0.184u^2v,\\
\psi_1(u,v)\!&=\!0.033u\!+\!0.096v\!+\!0.0760u^2\!+\!0.0534uv\!+\!0.094v^2\\
&+\!0.013uv^2\!-\!0.011v^3.
\end{split}\end{equation}
This overapproximation and the polynomials are illustrated in Figure \ref{fig:psifuncties}.
\begin{figure}[thb]\begin{center}
\includegraphics[width=0.3\textwidth,angle=0]{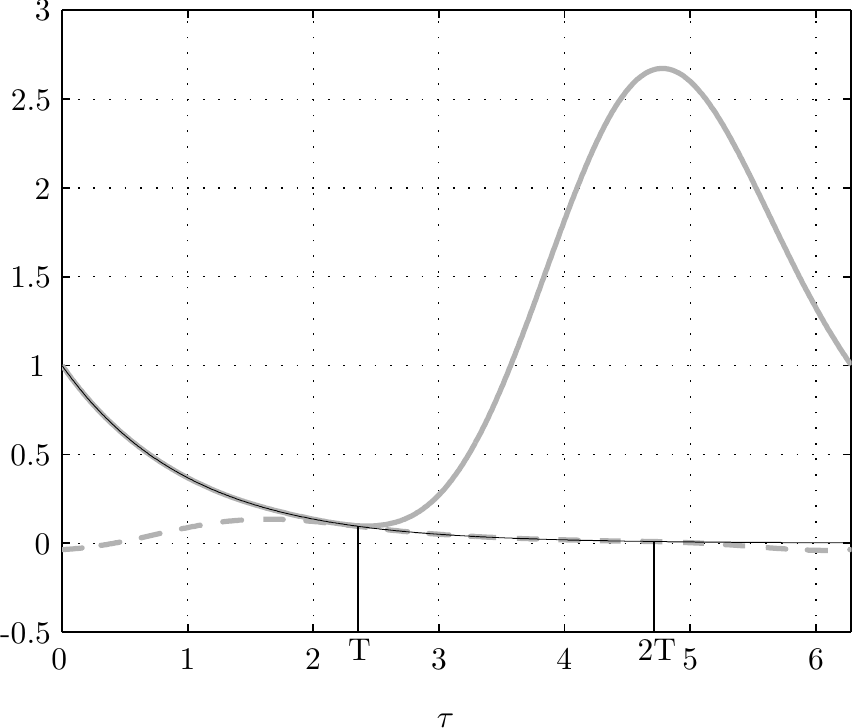}
\caption{Functions $e^{-\tau}$ (solid black), $\psi_0$ (solid grey), and $\psi_1$ (dashed grey).}
\label{fig:psifuncties}
\end{center}\end{figure}
If one is satisfied with an overapproximation accuracy of $\varepsilon\!=\!e^{-1.5\pi}\!\approx\!0.009$, then one can use this precomputed overapproximation (for the case $\theta=1$). If it is desired to have a simpler overapproximation (with less regions and polynomials of lower degree) or an even tighter approximation with $\varepsilon<0.009$, one can run Algorithm~\ref{alg} to obtain it.

\begin{remark} A few comments  are in order:
\begin{itemize}
\item The reason to take the functions $\phi_l$ in step 3 as smooth as possible ($m$ in $C^m$ as large as possible) is that a lower degree $K_l$ is needed to satisfy \eqref{eq:a}. Indeed, when $\phi_l\in C^m$ then the Fourier coefficients satisfy $k^ma_k\!\rightarrow\!0$, $k^mb_k\!\rightarrow\!0$ when $k\!\rightarrow\!\infty$ and thus the approximation error $|\sum_{k=0}^{K_l} [a_k \cos(k\theta\tau) + b_k \sin (k\theta\tau) ]-e^{-\tau}|=|\sum_{k=K_l}^{\infty} [a_k \cos(k\theta\tau) + b_k \sin (k\theta\tau) ]|$ on $I_l$ is smaller than $\varepsilon$ for smaller values of $K_l$.
\item A sufficiently high degree of $K_l$ such that \eqref{eq:a} holds can be obtained by increasing $K_l$ incrementally until \eqref{eq:a} is satisfied. If one is satisfied with an overapproximation accuracy of $\varepsilon\!=\!e^{-1.5\pi}$, then the precomputed overapproximation \eqref{precomp} with $N=2$ and $K_l=3$ can be used in case $\theta\!=\!1$.
\item Instead of selecting $T$ a priori, we can also select a maximal degree $K$ of the approximation functions $\psi_l$ in the sense that $K_l\leq K$ for all $l=0,1,\ldots,N-1$. Instead of increasing the degrees of the approximation functions $\psi_l$, one now can split the time interval $[0,-\ln \varepsilon)$ into smaller pieces until $|\psi_l(\cos \theta\tau,\sin \theta\tau) - e^{-\tau}|\leq \varepsilon$ for all $\tau \in I_l$ is satisfied for the fixed (low) degree $K$. This might lead to more regions (a larger $N$). This indicates that there is a trade-off between $N$ (number of basic semialgebraic sets in the overapproximation $\Fa$) and $K$ (the maximal degree of the Fourier series approximation). The smaller $N$ the higher $K$ and vice versa. However, note that the example of the overapproximation with $N=2$ and $K=3$ already provides a very tight approximation of $\varepsilon\!=\!e^{-1.5\pi}\!\approx\!0.009$ in case $\theta\!=\!1$.
\item The reason for splitting the set $\Fo$ in the $\tau=\frac{t}{m}$ direction is that the exponential function $\lambda\!=\!e^{-\tau}$ can generally not be $\varepsilon$-close approximated with the basis functions $\cos(\theta\tau)$ and $\sin(\theta\tau)$ in the interval $[0,- \ln \varepsilon]$ (unless $\varepsilon$ is chosen relatively large or $\theta$ is very small and thus $T$ can be selected larger). The additional scaling~$\theta$ can be used to adjust the period of $\cos(\theta\tau)$ and $\sin(\theta\tau)$ thereby offering a trade-off between the required amount of intervals $N$ and the order of the polynomial $y(u,v,\lambda)$ (provided the integer requirement on $\overline{\beta}_i$ is fulfilled). Indeed, $N\!=\!1$ in Definition~\ref{def:overapproximation} can be obtained by choosing $\theta$ sufficiently small such that $e^{-\tau}$ can be $\varepsilon$-close approximated using $u\!=\!\cos(\theta\tau)$ and $v\!=\!\sin(\theta\tau)$ in the interval $[0,- \ln \varepsilon]$ for any $\varepsilon$. However, this implies that the polynomial order $\overline{\beta}_i$ (see \eqref{eq:exacttimecomp3}) increases due to the relation $\beta_i\!=\!\frac{\theta\bar{\beta}_i}{m}$, which is undesirable from a computational point of view. On the other hand, if $\theta$ is chosen as large as possible (while respecting the integer requirement on $\overline{\beta}_i$), the order of the polynomial $y(u,v,\lambda)$ in \eqref{eq:exacttimecomp3} is minimal but a large amount of regions $N$ could be required.  
\end{itemize}
\label{remark:ding}\end{remark}
Using the $\varepsilon$-close overapproximation $\Fa$, the conditions \eqref{eq:positivepolymultpolvar} again are a special case of Problem~\ref{pr:opt} with $z\!=\!q$, $x\!=\!(u,v,\lambda)$, $\ell\!=\!2,~m\!=\!0$, $p(z,x)\!=\!0$, $g_1(z,x)\!=\!P_5(q,u,v,\lambda)$, $g_2(z,x)\!=\!P_6(q,u,v,\lambda)$, and $\mathcal{D}_l\!=\!\mathcal{F}_l,~l\!=\!0,\ldots,N$. However, since the polynomial positivity constraints \eqref{eq:positivepolymultpolvar} are now multivariate (meaning polynomials in more than one variable) instead of univariate, the equivalent LMI expression from \cite{Henrion2005} is not applicable. In general, checking positivity of a multivariate polynomial on a basic semialgebraic set is a hard problem, but it can often be approximated as closely as desired by a hierarchy of convex relaxations
\cite{Laurent2009,Lasserre2009}.
In this case, to make Problem~\ref{pr:opt} computationally tractable, the inequality conditions from \eqref{eq:generalpolprob} are replaced by stronger conditions in terms of primal moment and dual sums-of-squares (SOS) problems
to formulate a hierarchy of upper bounds on the minimum in Problem~\ref{pr:opt} that converge in the limit to the real minimum.
We refer the reader to \cite{Laurent2009,Lasserre2009}
for the conversion techniques to obtain the LMIs and further details.
There exist software packages, such as GloptiPoly \cite{Henrion2003}, SOSTOOLS \cite{Prajna2002} or YALMIP \cite{Yalmip},
that automatically build up a hierarchy of LMI relaxations, whose associated monotone sequence of optimal values converges to the global optimum.
Numerical certificates of optimality are also available, in terms of ranks of embedded moment matrices,
see \cite{Laurent2009,Lasserre2009}. 

\section{Control design procedure }\label{sec:propconroldes}

To summarize the previous design setup, suppose that the $d$-minimal controller \eqref{eq:tdccondc} for system \eqref{eq:timdomcoplant} has been designed via the Diophantine equation \eqref{eq:Diophantine} such that the desired closed-loop pole locations are achieved. Also suppose that more closed-loop poles are assigned than twice the number of poles of the plant. Then, according to \eqref{eq:qdegreeconstaint} there is additional control design freedom parameterized in the form of the Youla-Ku\v{c}era parameter, which can be used to shape the closed-loop time response. Time-domain constraints can be imposed using one of the  proposed relaxations in Section \ref{sec:extcompoles}, which yield the polynomial constraints in \eqref{eq:generalpolprob}. Next to constraint satisfaction,  \eqref{eq:generalpolprob} allows also the minimization of an objective function $p(z)$. This minimization can be exploited to obtain additional desired properties of the response  in terms of the design parameters $q$ and thus $y_i,a_i$ and $b_i$.  Consider the step response for example, which can be written as \eqref{eq:partialfractionaldecomrealcomp} with $p_0\!=\!0$ and where $y_0$ is the steady-state solution. Desirable properties of the unit step response are, for instance, a zero steady-state error, a small settling-time and small overshoot.
These properties can be accommodated in $p(z)$ as follows
\begin{description}
\item[Small steady-state error:] Set $p(z)\!=\!(1\!-\!y_0)^2$ to minimize the steady-state error.
\item[Short settling-time:] Set $p(z)\!=\!a_i^2\!+\!b_i^2$, where index $i$ corresponds to a slow mode in \eqref{eq:exacttimecomp}, to minimize the contribution of this mode, which improves the settling time. Alternatively, exponentially decreasing constraints can be specified that directly impose a certain desired settling behavior.
\item[Overshoot minimization:] In case one is interested in constraining the response by a fixed constant, e.g. constraining the overshoot, the exponential bounds relaxation is not suitable. This is due to the fact that the peak values of the systems response often occur in the time interval where the exponential upper and lower bounds are still very far from the actual signal. As opposed to the exponential bounds relaxation, the multivariate polynomial relaxation from Section~\ref{sec:mulpolrelax} is typically suited to minimize the overshoot of a step response by constraining the response \eqref{eq:exacttimecomp3pol} as $y(u,v,\lambda)\!\leq\!\gamma$ and specify $p(z)\!=\!\gamma$ to minimize the overshoot.
\end{description}

Of course, one can combine the above objectives in $p(z)$ using suitable weighting factors. Furthermore, the proposed relaxations enable the incorporation of the extensions that were mentioned in Section \ref{sec:henlmipospol} (e.g. related to responses to disturbances) in case the poles of the Laplace transform of the corresponding signal are complex. Setting up the optimization problem \eqref{eq:generalpolprob} by including the time-domain constraints on closed-loop signals using one of the proposed relaxations and defining the objective function $p(z)$ provides a systematic manner for obtaining linear controllers with desirable properties. This design framework will be illustrated in the next section.

\section{Numerical example}\label{sec:example}
We start with a simple simulation example to illustrate the efficiency of the proposed design method. Consider the simple model given by 
\begin{equation}
P(s)=\frac{y(s)}{u(s)}=\frac{1}{s+1}.
\end{equation}
The control objective is to let $y$ track a step reference from $0$ to $1$ as close as possible. Moreover, the controller \eqref{eq:tdccondc} will be designed such that the assigned complex closed-loop poles are $p_{1,2}\!=\!-1\pm2j$, $p_{3,4}\!=\!-2\pm4j$.
This is done by solving the Diophantine equation \eqref{eq:Diophantine} leading to the $d$-minimal controller
\begin{equation}
C(s)=\frac{d_0(s)}{c_0(s)}=\frac{68}{s^3+5s^2+28s+32},
\label{eq:origcontroller}\end{equation}
resulting in the closed-loop system given by the complementary sensitivity function
\begin{equation}
T(s)=\frac{68}{s^4+6s^3+33s^2+60s+100}.
\label{eq:origcontrollerclosedloop}\end{equation}
Using the Youla-Ku\v{c}era parameter $q(s)$ and realizing that according to \eqref{eq:qdegreeconstaint} we have $\deg q(s)\!\leq\! 2$, i.e., $q(s)\!=\!q_0\!+\!q_1s\!+\!q_2s^2$, the set of allowable controllers assigning the specified closed-loop poles is parameterized as
\begin{equation}\begin{split}
C(s)&=\frac{d(s)}{c(s)}=\frac{d_0(s)-a(s)q(s)}{c_0(s)+b(s)q(s)}\\
&=\frac{68-(q_0+(q_0+q_1)s+(q_1+q_2)s^2+q_2s^3)}{s^3+5s^2+28s+32+(q_0+q_1s+q_2s^2)},
\end{split}\end{equation}
resulting in the set of closed-loop transfer functions
\begin{equation}
T(s)=\frac{68-(q_0+(q_0+q_1)s+(q_1+q_2)s^2+q_2s^3)}{s^4+6s^3+33s^2+60s+100}.
\label{eq:exampleclosedlooptrans}\end{equation}
The Laplace transform of the time response of \eqref{eq:exampleclosedlooptrans} to a step input is then parameterized as
\begin{equation}\begin{split}
y(s)&=\frac{1}{s}T(s)=\\
&\frac{68-(q_0+(q_0+q_1)s+(q_1+q_2)s^2+q_2s^3)}{s(s+1+2j)(s+1-2j)(s+2+4j)(s+2-4j)}.
\label{eq:closedtimestepex}\end{split}\end{equation}
The corresponding partial fractional decomposition is equal to
\begin{equation}\begin{split}
y(s)=&\frac{y_0}{s}+\frac{a_1+jb_1}{s+1+2j}+\frac{a_1-jb_1}{s+1-2j}+\frac{a_2+jb_2}{s+2+4j}\\
&+\frac{a_2-jb_2}{s+2-4j}
\label{eq:closedtimestepexparfrac}\end{split}\end{equation}
where $y_0,a_1,b_1,a_2,b_2$ can be solved from the linear system of equations
\small\begin{equation}
\left[ \begin {array}{ccccc} 100&0&0&0&0\\60&40&80&20&40\\33&48&16&18&16\\6&10&4&8&8\\1&2&0&2&0\end {array} \right]
\left[\begin{array}{c}y_0\\a_1\\b_1\\a_2\\b_2\end{array}\right]\!=\!
\left[ \begin {array}{c} 68\\0\\0\\0\\0\end {array} \right]\!+\!\left[ \begin {array}{ccc} -1&0&0\\-1&-1&0\\0&-1&-1\\0&0&-1\\0&0&0\end {array} \right]
\left[\begin{array}{c}q_0\\q_1\\q_2\end{array}\right],
\label{eq:linrelrealresex}\end{equation}
\normalsize
where $q_0,q_1,q_2$ are the free variables in the Youla-Ku\v{c}era parameter to shape the time response. The goal is to determine values of $y_0,a_1, b_1, a_2, b_2$ (via $q_0,q_1,q_2$) such that the closed-loop time response to the step input has a favorable shape. 
The LMI problems were modeled with YALMIP \cite{Yalmip} and solved with SeDuMi \cite{Sturm1999}.

\subsection{Using the exponential bounds relaxation}\label{sec:applicationexprelax}
The exponential bounds on the step response of the closed-loop system \eqref{eq:exampleclosedlooptrans} are given by
\begin{equation}\begin{split}
\bar y_{\text{upper}}(t)&=y_0+(2|a_1|+2|b_1|)e^{-t}+(2|a_2|+2|b_2|)e^{-2t},\\
\bar y_{\text{lower}}(t)&=y_0-(2|a_1|+2|b_1|)e^{-t}-(2|a_2|+2|b_2|)e^{-2t},
\end{split}\end{equation}
where $y_0,a_1,b_1,a_2,b_2$ are related to $q_0,q_1,q_2$ via \eqref{eq:linrelrealresex}. The goal of this relaxation is to determine $q_0,q_1,q_2$ such that
\begin{equation}
\begin{array}{ll}P_3(\lambda)=g_u(\lambda)-y_{\text{upper}}(\lambda)&\geq0,\\P_4(\lambda)=y_{\text{lower}}(\lambda)-g_l(\lambda)&\geq0,\end{array} \label{eq:positivepolyexpboundsaexample}\end{equation}
for appropriately chosen {\em time-varying} bounds related to $g_u(\lambda)$ and $g_l(\lambda)$. Note that  $\lambda\!=\!e^{-t}$ and
\begin{equation}\begin{split}
y_{\text{upper}}(\lambda)&=y_0+(2|a_1|+2|b_1|)\lambda+(2|a_2|+2|b_2|)\lambda^2,\\
y_{\text{lower}}(\lambda)&=y_0-(2|a_1|+2|b_1|)\lambda-(2|a_2|+2|b_2|)\lambda^2.\\
\end{split}\label{eq:expboundsites}\end{equation}

To define the {time-varying} bounds, we first consider suitable upper and lower bounds on the step response that correspond to the $d$-minimal controller \eqref{eq:origcontroller} (with $q\!=\!0$). These are given by
\begin{equation}\begin{split}
g_{u_{\text{original}}}(\lambda)&=0.68+1.58\lambda+0.38\lambda^2,\\
g_{l_{\text{original}}}(\lambda)&=0.68-1.58\lambda-0.38\lambda^2,
\label{eq:origbounds}\end{split}\end{equation}
respectively. Based on these bounds, we now define tighter bounds $g_u(\lambda)$ and $g_l(\lambda)$ that are specified to be
\begin{equation}\begin{split}
g_u(\lambda)&=g_{u_{\text{original}}}(\lambda)+c_u=1.01+1.58\lambda+0.38\lambda^2\\
g_l(\lambda)&=g_{l_{\text{original}}}(\lambda)+c_l=0.99-1.58\lambda-0.38\lambda^2.
\end{split}\end{equation}
where $c_u\!=\!0.33$ and $c_l\!=\!0.31$ to guarantee a small steady-state error (smaller than $0.01$). The dominant term in the upper and lower bound in \eqref{eq:expboundsites} corresponds to the slow mode $e^{-t}$ with coefficient $(2|a_1|+2|b_1|)$. To minimize the contribution of this term and to bring the steady-state error close to zero, we specify the objective function according to Section~\ref{sec:propconroldes} as
\begin{equation}
p(q_0,q_1,q_2)=10(1-y_0)^2+2(a_1^2+b_1^2),
\label{eq:objectiveexpboundrel}\end{equation}
where $y_0$, $a_1$ and $b_1$ depend on $q_0,q_1,q_2$ as in \eqref{eq:linrelrealresex}. This results in the optimization problem
\begin{equation}
\begin{array}{ll}\underset{q_0,q_1,q_2}{\text{min}} &p(q_0,q_1,q_2)\\
\text{~~s.t.}&\eqref{eq:linrelrealresex}\\
&\eqref{eq:positivepolyexpboundsaexample}\quad \forall~\lambda\in[0,1].\end{array}
\label{eq:LMIproblem}\end{equation}
The Youla-Ku\v{c}era parameter resulting from the minimization problem \eqref{eq:LMIproblem} is
\begin{equation}
q(s)=-32.0-23.0s-3.0s^2,
\end{equation}
which yields the controller and closed-loop
\begin{equation}
C(s)=\frac{3s^3+26s^2+55s+100}{s^3+2s^2+5s},
\label{eq:newcontclosed}\end{equation}
\begin{equation}
T(s)=\frac{3s^3+26s^2+55s+100}{s^4+6s^3+33s^2+60s+100},
\label{eq:newcontclosed2}\end{equation}
together with the new bounds
\begin{equation}\begin{split}
g_{u_{\text{new}}}(t)&=1.00+1.25e^{-2t},\\
g_{l_{\text{new}}}(t)&=1.00-1.25^{-2t}.
\end{split}\end{equation}
This shows that the steady-state error is zero and the contribution of the slow mode has been completely eliminated.
The step responses together with their bounds of the original closed-loop system \eqref{eq:origcontrollerclosedloop} and of the new, optimized closed-loop system \eqref{eq:newcontclosed2} are depicted in Fig.~\ref{fig:stepresponsecompol12a}(a), while the Bode diagrams of the original controller \eqref{eq:origcontroller} and the new controller \eqref{eq:newcontclosed} are shown in Fig.~\ref{fig:stepresponsecompol12a}(b).
\begin{figure}[thb]\begin{center}
\subfigure[] {\includegraphics[angle=0]{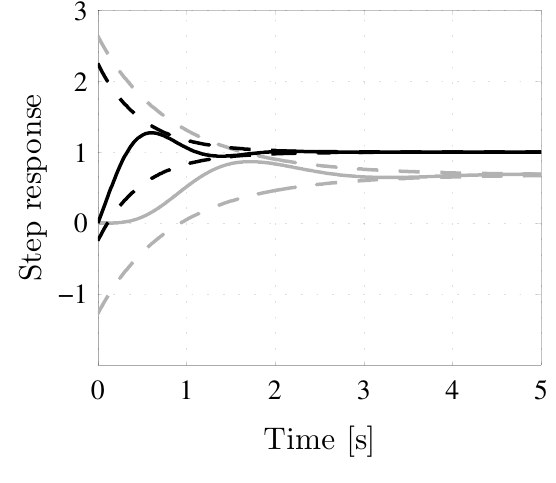}}
\subfigure[] {\includegraphics[angle=0]{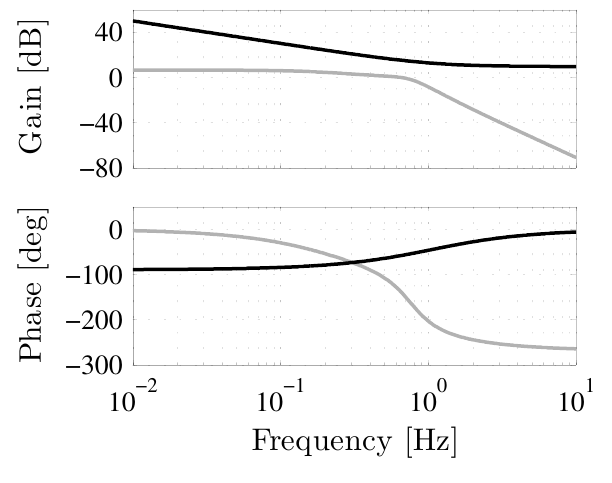}}
\caption{Results exponential bounds relaxation. (a)~New (solid black) and original (solid grey) step responses, original bounds (dashed grey), and new bounds (dashed black), (b)~Bode diagrams of the original (grey) and the new (black) controller.}
\label{fig:stepresponsecompol12a}
\end{center}\end{figure}
From Fig.~\ref{fig:stepresponsecompol12a}(a) it is obvious that the step response of the designed closed-loop system satisfies the specified bounds and additionally results in zero steady-state tracking error. Upon examination of Fig. \ref{fig:stepresponsecompol12a}(b) this can be explained by the fact that controller \eqref{eq:newcontclosed} exhibits an overall higher gain and hence results in a higher bandwidth of the closed-loop system resulting in a faster response, while it also implements integrating action providing the steady-state accuracy. Although there is some conservatism introduced by the fact that upper and lower bounds are used, this example demonstrates that a significant increase of the performance can be obtained using this method.
\begin{remark}
In this example the ability of the method to minimize the contribution of the slow mode using an additional objective function $p(z)$ in Problem~\ref{pr:opt} was demonstrated. Because of the specific objective function \eqref{eq:objectiveexpboundrel} the slow mode was completely cancelled within the controller. Therefore, for this specific example another way to arrive at controller \eqref{eq:newcontclosed} is to specify only two closed-loop poles $p_{3,4}\!=\!-2\pm4j$ (without $p_{1,2}\!=\!-1\pm2j$), resulting in $C(s)=\frac{3s+20}{s}$, which indeed is the minimal form of \eqref{eq:newcontclosed}. Interestingly, the optimization problem results in this controller in an automated and systematic manner.
\end{remark}

\subsection{Using the multivariate polynomial relaxation}\label{sec:multvarrelaxapplic}
In this section, the multivariate polynomial relaxation from Section~\ref{sec:mulpolrelax} is used to obtain suitable values of $q_0,q_1,q_2$ to improve the step response of closed-loop system \eqref{eq:exampleclosedlooptrans}. We have that $m\!=\!1,n_r\!=\!0,n_c\!=\!2$, $\overline{\alpha}_1\!=\!-1$, $\overline{\alpha}_2\!=\!-2$, $\overline{\beta}_1\!=\!2$ and $\overline{\beta}_2\!=\!4$. Let $\theta\!=\!1$ such that $u\!=\!\cos(\tau)$ and $v\!=\!\sin(\tau)$ so that \eqref{eq:exacttimecomp3} yields
\begin{equation}\begin{split}
y(t)&=\left((a_1+jb_1)\left(u+jv\right)^{2}+ (a_1-jb_1)\left(u +jv\right)^{2}\right)\lambda\\
&+\left((a_2+jb_2)\left(u+jv\right)^{4}+ (a_2-jb_2)\left(u +jv\right)^{4}\right)\lambda^2\\
&=\left(2a_1\left(u^2\!-\!v^2\right)+2b_12uv\right)\lambda\\
&+\left(2a_2\left(u^4\!+\!v^4\!-\!6u^2v^2\right)+2b_1\left(4vu^3\!-\!4uv^3\right)\right)\lambda^2.
\label{eq:exacttimecomp4}
\end{split}\end{equation}
As a consequence, for this example we obtain
\begin{equation}\begin{split}
w_1(u,v)&=u^2-v^2,\quad r_1(u,v)=2uv\\
w_2(u,v)&=u^4+v^4-6u^2v^2,\quad r_2(u,v)=4vu^3-4uv^3.
\end{split}\label{eq:qiriappen}\end{equation}
yielding the time response
\begin{equation}\begin{split}
y(u,v,\lambda)=&y_0+(2a_1(u^2-v^2)+4b_1uv)\lambda\\
&+(2a_2(u^4+v^4-6u^2v^2)+8b_2(vu^3-uv^3))\lambda^2,
\label{eq:maxtimel}\end{split}\end{equation}
which is a multivariate polynomial with 3 independent variables ($u,v,\lambda$) and three decision variables ($q_0,q_1,q_2$). Note that $(u,v,\lambda)\!\in\!\Fo$. Since $\Fo$ is not the finite union of a basic semialgebraic set, we can use Algorithm \ref{alg} or use the precomputed overapproximation in Section~\ref{sec:mulpolrelax} to obtain an $\varepsilon$-close overapproximation $\Fa$ of $\Fo$. We use here the precomputed overapproximation $\Fa$ with  $\varepsilon\!=\!e^{-1.5\pi}\!\approx\!0.009$. In accordance with Section~\ref{sec:propconroldes}, we formulate the problem as to find $q$ such that the overshoot $\gamma$ is small and that the steady-state error of the step response is minimized. Therefore, the problem is posed as
\begin{equation}
\begin{array}{ll}\underset{q_0,q_1,q_2}{\text{min}}&10(1-y_0)^2+\gamma\\\text{~~s.t.}&\eqref{eq:linrelrealresex}\\ &\gamma-y(u,v,\lambda)\geq0\quad \forall (u,v,\lambda)\in \Fa.\end{array}
\end{equation}
Rewriting this optimization problem  gives
\begin{equation}
\begin{array}{ll}\underset{q_0,q_1,q_2}{\text{min}}&10(1-y_0)^2+\gamma\\\text{~~s.t.}&\eqref{eq:linrelrealresex}\\ &\gamma-y(u,v,\lambda)\geq0\quad \forall (u,v,\lambda)\in \F_0\\
&\gamma-y(u,v,\lambda)\geq0\quad \forall (u,v,\lambda)\in \F_1\\
&\gamma-y(u,v,\lambda)\geq0\quad \forall (u,v,\lambda)\in \F_2,\end{array}
\label{eq:totlaproblemultvar}\end{equation}
with $\F_{0,1}$ as in \eqref{eq:Fltotal} and $\F_2$ as in \eqref{eq:FltotalN}.
This optimization problem is then solved with a hierarchy of LMI relaxations, as explained at the end of Section
\ref{sec:mulpolrelax}. The size of the resulting LMI problem depends on the order of the relaxation, and this can be used
as a tuning knob to adjust the trade-off between the desired accuracy and the computational complexity, see Section~\ref{sec:henrionreview}. Although it is a priori not clear
what is the order of the relaxation to arrive at the global minimum of $\gamma$, a heuristic method is to increase the order until not much improvement in the relaxed optimum $\tilde{\gamma}$ is observed anymore or until one is satisfied with the obtained value of $\tilde{\gamma}$. The obtained minimum values of $\tilde{\gamma}$ for various orders of
relaxation are given in Table~\ref{table:dinges1}. We used this heuristic approach for illustration purposes only. It is recommended to use the more systematic approach that is implemented in GloptiPoly \cite{Henrion2003} to arrive at the global minimum, and certify it numerically.
\begin{table}[htb]
\addtolength{\tabcolsep}{-2pt}
\centering
\begin{tabular}{|l|c|c|c|c|c|c|}
\hline
&\multicolumn{6}{c|}{order of LMI relaxation} \\
&1&2&3&4&5&10\\
\hline
$\tilde{\gamma}$ & 297.170 & 1.235 & 1.235 & 1.0718 & 1.0718 & 1.0718  \\
\hline
\end{tabular}
\vspace*{1mm}\caption{Upper bound $\tilde{\gamma}$ for various orders of LMI relaxation.}
\label{table:dinges1}\end{table}
Based on the figures in this table we expect that the global minimum of $\gamma$ is equal to $1.0718$, representing an overshoot of $7.18\%$. The corresponding Youla-Ku\v{c}era parameter is given by
\begin{equation}
q(s)=-32.0-17.0607s-3.0227s^2, 
\end{equation}
which yields the controller
\begin{equation}
C(s)=\frac{3.0s^3+20.0s^2+49s+100}{s^3+2.0s^2+10.9s}.
\label{eq:newcontclosedmultvar}\end{equation}
The step responses of both the original closed loop with the $d$-minimal controller \eqref{eq:origcontroller} and of the closed loop with controller \eqref{eq:newcontclosedmultvar} are depicted in Fig.~\ref{fig:newcontrollermultvara}, which shows a significant improvement as expected.
\begin{figure}[thb]\begin{center}
\includegraphics[angle=0]{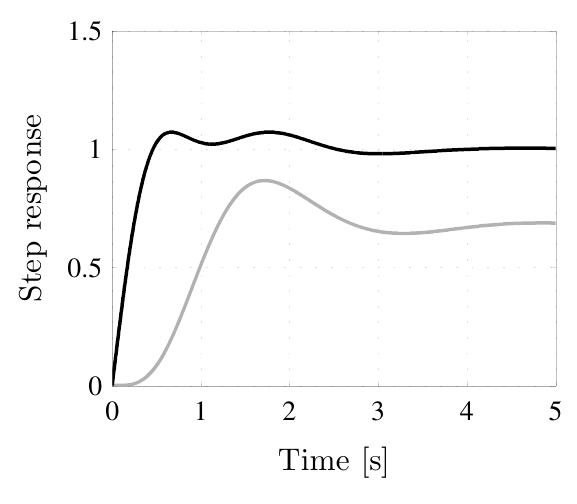}
\caption{New (black) and original (grey) step responses.}
\label{fig:newcontrollermultvara}
\end{center}\end{figure}
The maximum of the step response $y(t)$ equals $1.0714$ (i.e., 7.14\% overshoot), which indeed is $\varepsilon$-close to $\tilde{\gamma}\!=\!1.0718$.

This example showed that after controller design by pole placement it is possible to shape the transient time response of the system by assigning zeros to the closed-loop system through a suitable extension of the controller under pole invariance.

\section{Conclusions}\label{sec:henryconcl}
In this paper we provided a generally applicable design framework to obtain linear controllers for linear systems subject to time-domain constraints. In order to arrive at this design framework we extended recent results in \cite{Henrion2005} that applied only in case of {\em real} closed-loop poles and external signals (e.g.~references or disturbances) having Laplace transforms with {\em real} poles.  The design method is based on synthesizing linear controller via a closed-loop pole placement method in which the additional design freedom in terms of the Youla-Ku\v{c}era parameter is used to satisfy time-domain constraints on  the closed-loop signals based on sum-of-square techniques.  In order to extend the method to the practically relevant case where both inputs and closed-loop systems with complex conjugate poles are allowed, we proposed two relaxations.

The first relaxation, called the exponential bounds relaxation, exploited exponential upper and lower bounds on the response to any Laplace transformable input. Although this gives rise to potential conservatism, an example showed that by prescribing polynomial time-domain bounds, the system's performance to a step input can be improved with respect to the settling-time and the steady-state error. The second relaxation, the multivariate polynomial relaxation, removed the potential conservatism completely as we formally proved that it can approximate the original problem with arbitrary accuracy. Using these relaxations, we indicated how a polynomial optimization problem could be set up, in which next to constraint satisfaction, we could also optimize certain important closed-loop properties such as overshoot, settling-time and steady-state error. The resulting optimization problem can be solved using sum-of-squares and convex programming methods. As a consequence, the provided design framework is systematic in nature, as was also illustrated using a numerical example.

\section*{Ackwnoledgements}
D. Henrion acknowledges support by project No. 103/10/0628 of the Grant Agency of the Czech Republic.  

\bibliographystyle{plain}        
\bibliography{ltibiblio}           

\end{document}